\documentclass{amsart}

\makeindex

\usepackage[all, cmtip]{xy}
\usepackage{amssymb}

\usepackage{algorithmic}
\usepackage{amsmath}

\usepackage{amsthm}

\usepackage{amscd}

\usepackage{amsfonts}

\usepackage{amssymb}

\usepackage[pdftex]{graphicx}

\usepackage{multirow}

\usepackage{color}

\usepackage{epic}

\usepackage{graphicx}

\usepackage{pgf,tikz}
\usetikzlibrary{arrows}

\newtheorem{theorem}{Theorem}[section]
\newtheorem{lemma}[theorem]{Lemma}
\newtheorem{proposition}[theorem]{Proposition}

\newtheorem{question}[theorem]{Question}
\theoremstyle{corollary}
\newtheorem{corollary}[theorem]{Corollary}

\theoremstyle{definition}     % italic or bold etc.
\newtheorem{definition}[theorem]{Definition}

\newtheorem{example}[theorem]{Example}

\newtheorem{claim}[theorem]{Claim}
\theoremstyle{remark}
\newtheorem{remark}[theorem]{Remark}

\numberwithin{equation}{section}

\newcommand{\C}{\mathbb{C}}

\newcommand{\R}{\mathbb{R}}
\newcommand{\Z}{\mathbb{Z}}
\newcommand{\N}{\mathbb{N}}
\newcommand{\Q}{\mathbb{Q}}
\def\P{\mathbb{P}}

\def\Pic{\operatorname{Pic}}
\def\Proj{\operatorname{Proj}}

\def\mult{\operatorname{mult}}

\def\Nef{\operatorname{Nef}}
\def\Eff{\operatorname{Eff}}

\def\cha{\operatorname{char}}

\input xy
\xyoption{all}

\title[Cox rings and redundant blow-ups]{Cox rings of rational surfaces and \\ redundant blow-ups}

\begin{document}

\author{DongSeon Hwang}
\address{Department of Mathematics, Ajou University, Suwon, Korea}
\email{dshwang@ajou.ac.kr}

\author{Jinhyung Park}
\address{Department of Mathematical Sciences, KAIST, Daejeon, Korea}
\curraddr{School of Mathematics, Korea Institute for Advanced Study, Seoul, Korea}
\email{parkjh13@kaist.ac.kr}

\thanks{DongSeon Hwang was partially supported by Basic Science Research Program through the National Research Foundation of Korea (NRF) funded by the Ministry of Education (2011-0022904).
Jinhyung Park was partially supported by TJ Park Science Fellowship for Ph.D Students.}

\subjclass[2010]{Primary 14J26; Secondary 14C20}

\date{\today}

\keywords{redundant blow-up, rational surface, Cox ring, Mori dream space, Zariski decomposition.}

\begin{abstract}
We prove that the redundant blow-up preserves the finite generation of the Cox ring of a rational surface under a suitable assumption, and we study the birational structure of Mori dream rational surfaces via redundant blow-ups. It turns out that the redundant blow-up completely characterizes birational morphisms of Mori dream rational surfaces with anticanonical Iitaka dimension $0$. As an application, we construct new Mori dream rational surfaces with anticanonical Iitaka dimension $0$ and $-\infty$ of arbitrarily large Picard number.
\end{abstract}

\maketitle

\tableofcontents

\section{Introduction}

Algebraic varieties with finitely generated Cox rings, or equivalently, Mori dream spaces, have attracted considerable attention for various purposes. As its name suggests, Cox rings decode lots of birational structures of Mori dream spaces in view of Mori theory. On the other hand, the universal torsor, which plays an important role in fining rational points (\cite{CTS}), of a Mori dream space can be explicitly calculated from the Cox ring. However, the classification of such varieties has been remained as a challenging problem, even in the case of rational surfaces.

Testa, V\'{a}rilly-Alvarado, and Velasco (\cite[Theorem 1]{TVV10}) and Chen and Schnell (\cite[Theorem 3]{CS08}) independently showed that the Cox ring of every big anticanonical rational surface, i.e., a smooth projective rational surface with big anticanonical divisor, is finitely generated.
Motivated by this result, Artebani and Laface (\cite{AL11}) investigated a rational surface $S$ with a finitely generated Cox ring according to its \emph{anticanonical Iitaka dimension}
$$\kappa(-K_S) := \max \{\dim \varphi_{|-nK_S|}(S) : n \in \N\},$$
whose value is one of  $2,1,0,$ and $-\infty$. In the case $\kappa(-K_S)=1$, the Cox ring of a rational surface $S$ is finitely generated if and only if the relatively minimal model of the elliptic fibration $\pi \colon S \rightarrow \P^1$ has a Jacobian fibration with a finite Mordell-Weil group  in characteristic zero (\cite[Theorem 4.2, Theorem 4.8, and Theorem 5.3]{AL11}).
In contrast, very little is known for the case $\kappa(-K_S) \leq 0$. Some studies have investigated a rational surface $S$ with a finitely generated Cox ring when $\kappa(-K_S)=0$, and only one such example (that is a Coble rational surface) was recently given by Laface and Testa (\cite[Theorem 6.3]{LT11}). However, to the best of the authors' knowledge, there was no known example for the case $\kappa(-K_S)=-\infty$.

The principal aim of the present paper is to propose a systematic way to study the classification of rational surfaces with finitely generated Cox rings in terms of redundant blow-ups. Here, we briefly introduce the notion of redundant blow-ups, basically developed in the authors' previous paper \cite{HP} motivated by Sakai's work (\cite{Sak84}). Let $S$ be a smooth projective rational surface with pseudo-effective anticanonical divisor, and let $-K_S = P + N$ be the Zariski decomposition. A point $p$ in $S$ is called a \emph{redundant point} if $\mult_p N \geq 1$, and the blow-up $f \colon \widetilde{S} \rightarrow S$ at a redundant point $p$ is called a \emph{redundant blow-up}. Note that the redundant blow-up preserves the anticanonical Iitaka dimension.

We first point out that the redundant blow-up plays a dominant role in studying morphisms between rational surfaces with finitely generated Cox rings. Even though there are  many types of morphisms between big anticanonical rational surfaces in general, every big anticanonical rational surface can be obtained by a sequence of redundant blow-ups from the minimal resolution of a del Pezzo surface with rational singularities (\cite[Proposition 4.1 and Theorem 4.3]{Sak84}). On the contrary, there are only two types of morphisms for the case $\kappa(-K)=1$ (see \cite[Lemma 4.4]{AL11} and Theorem \ref{1red} for more detail). Furthermore, it turns out that the redundant blow-up completely characterizes morphisms for the case $\kappa(-K) = 0$.

\begin{theorem}\label{chbl}
Let $f \colon \widetilde{S} \rightarrow S$ be a birational morphism of smooth projective rational surfaces with finitely generated Cox rings. If $\kappa(-K_{\widetilde{S}})=\kappa(-K_{S})=0$, then $f$ is a sequence of redundant blow-ups.
\end{theorem}

Recall that the Cox ring of a rational surface is finitely generated if and only if the effective cone is rational polyhedral and every nef divisor is semiample (\cite[Proposition 2.9]{HK00}). Most of the previous studies on the characterization of rational surfaces having finitely generated Cox rings is concentrated on the semiampleness of a nef divisor. Although the rational polyhedrality of the effective cone is already interesting in its own right and has been studied by many authors in a variety of flavors (see e.g., \cite{GM2}, \cite{H} ,\cite{LH}, \cite{Nik00}, \cite{Tot10}), it is still not well understood. The main difficulty is due to the fact that the blow-up in general changes lots of the structure of the effective cone. In this viewpoint, it is an interesting problem under what condition the blow-up preserves the finite generation of Cox rings. The following is the main theorem of this paper.

\begin{theorem}\label{redfg}
Let $f \colon \widetilde{S} \to S$ be a redundant blow-up at a point $p$ of a Mori dream rational surface with $\kappa(-K_S) \geq 0$, and let $-K_S = P+N$ be the Zariski decomposition.
Then the Cox ring of $\widetilde{S}$ is finitely generated unless $\kappa(-K_S)=0$ and $\mult_p N =1$.
\end{theorem}

To prove Theorem \ref{redfg}, we bound the number of all possible curves on $S$ whose strict transforms to $\widetilde{S}$ become negative curves.

\begin{remark}
The assumption in Theorem \ref{redfg} is necessary. More precisely, there exists a rational surface $S$ with $\kappa(-K_S)=0$ admitting a redundant blow-up $\widetilde{S} \rightarrow S$ such that the Cox ring of $S$ is finitely generated but the Cox ring of $\widetilde{S}$ is not (see Example \ref{nonmdrsex}).  On the other hand, one can also construct a  rational surface $S$ with $\kappa(-K_S)=0$ and $\mult_p N = 1$ such that the redundant blow-up $\widetilde{S} \rightarrow S$ at $p$ preserves the finite generation of Cox ring (see Remark \ref{-infrem}).
\end{remark}

Now, thanks to Theorems \ref{chbl} and \ref{redfg}, we can construct many new rational surfaces with finitely generated Cox rings by taking redundant blow-ups. In particular, we construct a series of infinitely many new examples of rational surfaces with finitely generated Cox rings for $\kappa(-K_S) = 0$ and $-\infty$. The following theorem will be shown by explicit construction (see Example \ref{-infty}).

\begin{theorem}\label{exthm}
For each $n \geq 11$, there exist smooth projective rational surfaces $S$ and $\widetilde{S}$ with finitely generated Cox rings such that
 \begin{enumerate}
  \item $\kappa(-K_S)=0$ and $\rho(S) = n$, and
  \item $\kappa(-K_{\widetilde{S}})=-\infty$ and $\rho(\widetilde{S}) = n+1$,
 \end{enumerate}
where $\rho$ denotes the Picard number.
\end{theorem}

All the surfaces in Example \ref{-infty} with $\kappa(-K)=0$ have effective anticanonical divisor, so they are not Coble rational surfaces. In fact, all those surfaces are obtained by blow-ups from extremal rational elliptic surfaces with effective $k^{*}$-action. By a result of Knop (\cite{Kn}), all rational surfaces with effective $k^*$-action have finitely generated Cox rings, and their Cox rings can be explicitly calculated by \cite[Theorem 1.3]{HS10} in characteristic zero.
%There are four extremal rational elliptic surfaces having effective $\C^{*}$-action (see Remark \ref{extell})

Finally, we propose another viewpoint on rational surfaces with finitely generated Cox rings.
Recall that the Cox ring of  the surface obtained by blowing up at least $9$  very general points on the projective plane is not finitely generated. However, the Cox ring of the surface obtained by blowing up at any number of points on a line on the projective plane is always finitely generated (\cite[Example 3.3]{EKW04}). Thus we observe that a rational surface obtained by blowing up at points in special position  has a tendency to have a finitely generated Cox ring. This observation motivates us to study the extremal case: finite generation of the Cox ring of the minimal resolution of a rational surface with Picard number one.

\begin{theorem}\label{minen}
Let $\bar{S}$ be a normal projective rational surface of Picard number one, and let $g \colon  S \rightarrow \bar{S}$ be its minimal resolution. Assume that $-K_{\bar{S}}$ is nef and $\bar{S}$ contains at worst log terminal singularities which are not canonical singularities. Then the Cox ring of $S$ is finitely generated.
\end{theorem}

As in the proof of Theorem \ref{redfg}, we directly control possible extremal rays of the effective cone. For this purpose, we assume that $\bar{S}$ does not contain canonical singularities, and in fact, such an assumption is necessary (see Remark \ref{Ohashi}).

When the base field is $\C$, a normal projective surface  with at worst quotient singularities and the second topological Betti number $b_2 = 1$ is called a \emph{$\Q$-homology projective plane}. There are examples of rational $\Q$-homology projective planes satisfying the condition of Theorem \ref{minen} in \cite{HKex}, and they admit redundant blow-ups. Accordingly, we can obtain more examples of rational surfaces with $\kappa(-K)=0$ having finitely generated Cox rings, which are not Coble rational surfaces (see Remark \ref{ncob}).

The remainder of this paper is organized as follows. In Section \ref{setupsec}, we briefly recall basic properties of redundant blow-ups and Cox rings. Section \ref{classifisec} is devoted to the investigation of the classification of Mori dream rational surfaces by focusing on redundant blow-ups. In particular, we prove Theorem \ref{chbl}. In Section \ref{fgsec}, we prove Theorem \ref{redfg}. In Section \ref{exs}, we construct infinitely many new rational surfaces having finitely genrated Cox rings with anticanonical Iitaka dimension $0$ as well as $-\infty$, which supports Theorem \ref{exthm}. Finally, in Section \ref{qhpps}, we consider the finite generation of Cox rings of minimal resolutions of rational $\Q$-homology projective planes, and we prove Theorem \ref{minen}.

Throughout the paper, we work over an algebraically closed field $k$ of arbitrary characteristic.

\subsection*{Acknowledgements}
The authors would like to thank Junmyeong Jang for useful comments on the techniques in positive characteristic, Michela Artebani and Antonio Laface for sending them the revised paper \cite{AL11}, and Damiano Testa for interesting discussion. DongSeon Hwang also thanks Hisanori Ohashi for useful discussion around the examples in \cite{HKO}.  Jinhyung Park wishes to express his deep gratitude to his advisor Sijong Kwak for warm encouragement.
The authors wish to thank the anonymous referee for the careful reading and the valuable suggestions, especially for the simplification of the proof of Lemma \ref{bd}.

\section{Preliminaries}\label{setupsec}

In this section, we collect basic notions and useful facts.

\subsection{Redundant blow-ups}
Let $S$ be a smooth projective rational surface, and let $D$ be a $\Q$-divisor. The \emph{Iitaka dimension} of $D$ is given by $$\kappa(D) := \max \{\dim \varphi_{|-nD| (S)} : n \in \N\},$$
whose value is one of $2,1,0,$ and $-\infty$. We call $\kappa(-K_S)$ the \emph{anticanonical Iitaka dimension} of $S$. Note that $\kappa(-K_S) \geq 0$ if and only if $-K_S$ is pseudo-effective (\cite[Lemma 3.1]{Sak84}).
We will frequently use the notion of the \emph{Zariski decomposition} of a pseudo-effective $\Q$-divisor $D$ (see \cite[Section 2]{Sak84} for details): $D$ can be written uniquely as $P+N$, where $P$ is a nef $\Q$-divisor, $N$ is an effective $\Q$-divisor, $P.N=0$, and the intersection matrix of the irreducible components of $N$ is negative definite if $N \neq 0$.

Let $S$ be a smooth projective rational surface with $\kappa(-K_S) \geq 0$, and let $-K_S = P + N$ be the Zariski decomposition. Let $f: \widetilde{S} \to S$ be a blow-up at a point $p$ in $S$ with the exceptional divisor $E$.

\begin{definition}\label{reddef}
A point $p$ is called \emph{redundant} if $\mult_p N \geq 1$. The blow-up $f: \widetilde{S} \rightarrow S$ at a redundant point $p$ is called a \emph{redundant blow-up}, and the exceptional curve $E$ is called a \emph{redundant curve}.
\end{definition}

Note that we always have $\kappa(-K_{\widetilde{S}})  \leq \kappa(-K_S)$ in general. If $f$ is a redundant blow-up, then $\kappa(-K_{\widetilde{S}}) = \kappa(-K_S) \geq 0$ by \cite[Lemma 6.9]{Sak84} and the following lemma.

\begin{lemma}[{\cite[Corollary 6.7]{Sak84}}]\label{redlem}
Assume that $\kappa(-K_{\widetilde{S}}) \geq 0$ so the we have the Zariski decomposition  $-K_{\widetilde{S}}=\widetilde{P}+\widetilde{N}$. Then the following are equivalent:
\begin{enumerate}
 \item $f$ is a redundant point.
 \item $\widetilde{P}= f^{*}P$ and $\widetilde{N}= f^{*}N - E$.
\end{enumerate}
\end{lemma}

For more basic properties of redundant blow-ups, we refer to \cite{HP}.

\subsection{Cox rings of rational surfaces}

Let $S$ be a smooth projective surface with $q(S)=h^1(\mathcal{O}_S)=0$. We define the Cox ring of $S$ as $\bigoplus_{L \in \Pic(S)}H^0(S, L)$. Then we have the following geometric characterization when the Cox ring of $S$ is finitely generated.

\begin{theorem}[{\cite[Proposition 2.9]{HK00}} and {\cite[Corollary 1]{GM}}]
Let $S$ be a smooth projective surface with $q(S)=0$. Then the Cox ring of $S$ is finitely generated if and only if the following hold:
\begin{enumerate}
\item the effective cone $\Eff(S)$ is rational polyhedral, and
\item every nef divisor is semiample.
\end{enumerate}
\end{theorem}

Such a surface is called a \emph{Mori dream surface}. The following lemma will be useful in proving that the effective cone is rational polyhedral. We remark that the proof in \cite{AL11} works for arbitrary characteristic, since the cone theorem still holds for positive characteristic by \cite[Theorem 4.4]{Ta}.

\begin{lemma}[{\cite[Corollary 2.2]{AL11}} and {\cite[Corollary 4.2]{LH}}]\label{cone}
Let $S$ be a smooth projective rational surface with $\kappa(-K_S) \geq 0$. Then $\Eff(S)$ is rational polyhedral if and only if there are only finitely many $(-1)$-curves and $(-2)$-curves on $S$.
\end{lemma}

The following lemma will be useful in checking that a nef divisor is semiample.

\begin{lemma}[{\cite[Lemma 3.1]{LT11}}]\label{nefsemiample}
Let $S$ be a smooth projective rational surface, and let $M$ be a nef divisor on $S$. If $-K_S.M>0$, then $M$ is semiample.
\end{lemma}

\section{Classification of Mori dream rational surfaces via redundant blow-ups}\label{classifisec}

In this section, we investigate the classification problem of Mori dream rational surfaces with $\kappa(-K) \geq 0$ via redundant blow-ups, and we prove Theorem \ref{chbl}. Let $S$ be a smooth projective rational surface with $\kappa(-K_S) \geq 0$.

\subsection{$\kappa(-K_S)=2$}

Recall that $S$ is always a Mori dream rational surface by \cite[Theorem 1]{TVV10} or \cite[Theorem 3]{CS08}. Moreover, we have the following classification result.

\begin{theorem}[{\cite[Proposition 4.2 and Theorem 4.3]{Sak84}}]\label{Sakai}
Every big anticanonical rational surface can be obtained by a sequence of redundant blow-ups from the minimal resolution of a del Pezzo surface with rational singularities.
\end{theorem}

For more basic properties of big anticanonical rational surfaces, we refer to \cite{Sak84} and \cite{HP}.

\subsection{$\kappa(-K_S)=1$}

In characteristic zero, we have the classification result as follows. By \cite[Theorem 3.4 and Theorem 4.8]{AL11}, every Mori dream rational surface with anticanonical Iitaka dimension $1$ can be obtained by a sequence of blow-ups of type (1) or (2) in Lemma \ref{1red} from a relatively minimal rational elliptic surface whose Jacobian fibration has a finite Mordell-Weil group. Note that by the Ogg-Shafarevich theory, every such rational elliptic surface can be obtained from extremal rational elliptic surfaces, which are completely classified in \cite{MP86}. 

\begin{theorem}[{\cite[Lemma 4.4]{AL11}}]\label{1red}
Let $f \colon \widetilde{S} \rightarrow S$ be the blow-up at a point $p$ in a smooth projective rational surface $S$ with $\kappa(-K_S)=1$. Then $\kappa(-K_{\widetilde{S}})=\kappa(-K_S)$ if and only if one of the following hold:
 \begin{enumerate}
	\item $f$ is a redundant blow-up, or
	\item $\mult_p N<1$ but $\mult_p (-K_S) > 1$.
 \end{enumerate}
For the second case, we have $\widetilde{P}=\frac{\mult_p(-K_X)-1}{\mult_p P}f^{*}P$ and $\widetilde{N}=\frac{1-\mult_p N}{\mult_p P}f^{*}P + f^{*}N - E$, where
$-K_{\widetilde{S}}=\widetilde{P} + \widetilde{N}$ and $-K_S = P+N$ are the Zariski decompositions, and $E$ is the exceptional divisor of $f$.
\end{theorem}

\begin{proof}
We note that the proof in \cite{AL11} works for arbitrary characteristic.
\end{proof}

See \cite[Definition 4.3]{AL11} for well-definedness of $\mult_p P$ and $\mult_p(-K_S)$. The Mori dream rational surface $S$ with $\kappa(-K_S)=1$ in Example \ref{-infty} admits infinitely many redundant blow-ups. In \cite[Proposition 4.6]{AL11}, Artebani and Laface constructed an infinite sequence of blow-ups of type (2) in Lemma \ref{1red}.

Furthermore, we can also prove that the blow-ups of type (1) and (2) of Lemma \ref{1red} preserve the finite generation of Cox rings. The following theorem was known in characteristic zero by \cite[Theorem 4.8]{AL11}.

\begin{theorem}\label{1bwfg}
Let $S$ be a Mori dream rational surface, and let $f \colon \widetilde{S} \rightarrow S$ be the blow-up at a point $p$. If $\kappa(-K_{\widetilde{S}})=\kappa(-K_S)=1$, then $\widetilde{S}$ is also a Mori dream rational surface.
\end{theorem}

We will prove this theorem for all characteristic in Section \ref{fgsec}.

\subsection{$\kappa(-K_S)=0$}
No classification theory is previously given in this case, but we present a characterization of blow-ups preserving the finite generation of Cox rings. The following theorem immediately implies Theorem \ref{chbl}.

\begin{theorem}
Let $S$ be a Mori dream rational surface with $\kappa(-K_S)=0$, and let $f \colon \widetilde{S} \rightarrow S$ be a blow-up at a point. Suppose that $\widetilde{S}$ is also a Mori dream rational surface. Then $\kappa(-K_{\widetilde{S}})=\kappa(-K_{S})$ if and only if $f$ is a redundant blow-up.
\end{theorem}

\begin{proof}
If $f$ is a redundant blow-up, then $\kappa(-K_{\widetilde{X}})=0$ by Lemma \ref{redlem}. Conversely, assume that $\kappa(-K_{\widetilde{S}})=0$. Let $-K_S=P+N$ be the Zariski decomposition. Since $\kappa(-K_S)=\kappa(P)=0$ and the nef divisor $P$ on a Mori dream rational surface $S$ is semiample, we have $P=0$. Let $E$ be the exceptional divisor of $f$. Then for the Zariski decomposition $-K_{\widetilde{S}} = \widetilde{P} + \widetilde{N}$, we have $\widetilde{P}=0$ and $\widetilde{N}=-K_{\widetilde{S}}= f^{*} (N) - E$. By Lemma \ref{redlem}, $f$ is a redundant blow-up.
\end{proof}

In view of Theorem \ref{chbl}, we need to classify Mori dream rational surfaces $S$ with $\kappa(-K_S)=0$ such that for every blow-down $S \to S'$, we have $\kappa(-K_{S'}) \geq 1$. To understand the structure of Mori dream rational surfaces with $\kappa(-K)=0$, we construct many new examples of such surfaces and redundant blow-ups (see Sections \ref{exs} and \ref{qhpps}).

We have seen that there is a tendency of morphisms between Mori dream rational surfaces with the same anticanonical Iitaka dimension $\kappa(-K) \geq 0$ to preserve the positive part of the Zariski decomposition of the anticanonical divisor when $\kappa(-K)$ decreases, and finally, there are only redundant blow-ups for the case $\kappa(-K)=0$. In this viewpoint, we expect that morphisms between Mori dream rational surfaces with $\kappa(-K)=-\infty$ are very rare. It is an interesting problem to characterize such morphisms.

\section{Finite generation of Cox rings under redundant blow-ups}\label{fgsec}

In this section, we prove Theorem \ref{redfg} and Theorem \ref{1bwfg}. First, we prove some useful lemmas.

\begin{lemma}\label{elfib}
Let $S$ be a smooth projective surface with $q(S)=0$, and let $M$ be an effective base point free divisor on $S$ such that $M^2 =0$ and $M.K_S = 0$. If $M$ is nontrivial, then the linear system $|M|$ induces an elliptic (or quasi-elliptic) fibration $\varphi_{|M|} \colon S \rightarrow \P^1$.
\end{lemma}

\begin{proof}
We slightly modify \cite[Proof of Proposition 1.3]{AL11} to make it work for arbitrary characteristic. Since $M$ is not big and is base point free, we have $\kappa(M)=1$. Thus the morphism $\varphi_{|M|} \colon S \rightarrow B \subset \P^n$ maps to an integral curve. Consider the Stein factorization
\[
\xymatrix{
S \ar[rr]^{\varphi_{|M|}} \ar[rd]_f && B\\
&B' \ar[ru]_g &}
\]
where $B'$ is an integral curve, $f$ is a contraction, and $g$ is a finite morphism. We have the Leray spectral sequence
$$E_2^{p,q}=H^p(B', R^q f_{*} \mathcal{O}_S) \Rightarrow H^{p+q}(S, \mathcal{O}_S).$$
Since the edge homomorphism $H^1(B', \mathcal{O}_{B'})=E_2^{1,0} \rightarrow H^1(S, \mathcal{O}_S)$ is injective, we obtain $0=q(S) \geq p_a(B')$. Thus $B' \simeq \P^1$.  Let $a:=\deg(g) \deg_{\P^n}(B)$. Then $M = aD$, where $\mathcal{O}_S (D) = f^{*} \mathcal{O}_{\P^1} (1)$. Since $h^0 (D) \geq 2$, we have
$$n+1 = h^0(M) = h^0(aD) \geq a+1.$$
On the other hand, since the embedding $B\subset \P^n$ is non-degenerate, $a \geq \deg_{\P^n}(B) \geq n$. Thus $a=\deg_{\P^n}(B) = n$, and hence, $B$ is a curve of minimal degree. By the classification of varieties of minimal degree (see \cite[Theorem 1]{EH}), we conclude that $g$ is a Veronese embedding, and hence, we get a fibration $\varphi_{|M|} \colon S \rightarrow B \simeq \P^1$ with connected fibers. Note that the condition $(\varphi_{|M|})_* \mathcal{O}_S = \mathcal{O}_{\P^1}$ is equivalent to the function field $k(\P^1)$ being algebraically closed in $k(S)$. By \cite[Corollary 7.3]{B01}, all but finitely many fibers of $\varphi_{|M|}$ are integral curves. Now, the assertion immediately follows from the adjunction formula.
\end{proof}

The following is an important technical ingredient of the proof of Theorem \ref{redfg}.

\begin{lemma}\label{bd}
Let $S$ be a smooth projective surface such that the nef cone $\Nef(S)$ is a rational polyhedral cone. Then for any fixed integer $r \geq 0$, there is only a finite number of classes $[C] \in \Eff(S)$ such that $C^2 = r^2 -1, K_S .C = -r-1,$ and $C$ is an integral curve.
\end{lemma}

\begin{proof}
Since $\Eff(S)$ is dual to $\Nef(S)$, there are only finitely many ($-1$)-curves. Thus we assume that $r \geq 1$. 

First, we consider the case $r=1$. Let $C$ be an integral curve with $C^2=0$. We claim that $[C]$ generates an extremal ray of the nef cone $\Nef(S)$. Suppose that $[C] = x+y$ with $x, y \in \Nef(S)$. We have
$$0=C^2 = x^2 + 2x.y+y^2.$$
Since $x$ and $y$ are nef classes, it follows that $x^2=x.y=y^2=0$. Let $A$ be an ample divisor on $S$. We denote by $a:=[A].x$ and $b:=[A].y$. We have $[A].(bx-ay)=0$ and $(bx-ay)^2=0$. By the Hodge index theorem, $x=y$. We have shown the claim. Note that there are at most one class $u$ in each extremal ray of $\Nef(S)$ such that $[K_S].u=-2$. Since we assume that $\Nef(S)$ is rational polyhedral, there are finitely many classes $[C]\in \Eff(S)$ such that $C^2 = 0, K_S .C = -2,$ and $C$ is an integral curve.

It remains to consider the case $r \geq 2$. Let $C$ be an integral curve with $C^2=r^2-1$.
Since $C^2 =r^2-1> 0$, it follows that $C$ is nef and big. Let $A_1, \ldots, A_k$ be nontrivial $\Z$-divisors generating the nef cone $\Nef(S)$. We may write $C = \sum_i a_i A_i$, where $a_i \in \Q_{\geq 0}$ for every $i$. It suffices to show that every class $[C]$ lies in some bounded region in the N\'{e}ron-Severi space $N^1 (S)_{\R}:=N^1(S) \otimes \R$, because $[C]$ is an integral point. By the Hodge index theorem, $A_i.C \geq 1$ for every $i$. We have
$$ r^2 -1= C^2 = \sum a_i A_i.C \geq \sum a_i,$$
and hence, every $a_i$ is bounded. Thus $[C]$ lies in some bounded region in $N^1 (S)_{\R}$.
\end{proof}

Now, we present a proof of Theorem \ref{1bwfg}.

\begin{proof}[Proof of Theorem \ref{1bwfg}]
Let $-K_S = P+N$ and $-K_{\widetilde{S}}=\widetilde{P} + \widetilde{N}$ be the Zariski decompositions.
By Lemma \ref{elfib}, for some integer $m>0$, the linear system $|m \widetilde{P}|$ induces a fibration $\varphi \colon \widetilde{S} \rightarrow \P^1$. First, we show that the effective cone $\Eff(\widetilde{S})$ is rational polyhedral. By Lemma \ref{cone}, it is sufficient to show that there are only finitely many ($-1$)-curves and ($-2$)-curves on $\widetilde{S}$. Let $\widetilde{C}$ be a ($-2$)-curve. Since $0=-K_{\widetilde{S}}.\widetilde{C}=\widetilde{P}.C + \widetilde{N}.C$, the curve $\widetilde{C}$ is contained in either a reducible fiber of $\varphi$ or the support of $\widetilde{N}$. Thus $\widetilde{S}$ contains finitely many ($-2$)-curves.

Let $\widetilde{C}$ be a ($-1$)-curve on $\widetilde{S}$ different from the exceptional divisor $E$ of $f$. In this case, we use an idea from \cite[Proof of Theorem 4.8]{AL11}. By \cite[Theorem 3.4]{Sak84}, there exists a birational morphism $g \colon S \rightarrow S_0$ such that $S_0$ is a relatively minimal model of the fibration $\varphi_{|P|} \colon S \rightarrow \P^1$. Then $-K_{S_0} = P_0$ is the Zariski decomposition. By Lemma \ref{1red}, $\widetilde{P}=\alpha f^*g^* P_0$, where $\alpha$ is a rational number with $0 < \alpha <1$. We have $-K_{\widetilde{S}}.\widetilde{C}=1$. Then either $\widetilde{C}$ is a component of $N$, a reducible fiber of $\varphi$, or $0< \widetilde{C}.\widetilde{P} \leq 1$. For the first two cases, there are only finitely many possibilities. Thus assume that we are in the third case. We can further assume that $\widetilde{C}$ is not a component of the exceptional divisor of $g \circ f \colon \widetilde{S} \rightarrow S_0$. We have
$$r=\widetilde{C}.E \leq \widetilde{C}.(f^{*}g^{*}(-K_{S_0}) + K_{\widetilde{S}}) = \frac{1}{\alpha}\widetilde{C}.\widetilde{P} -1 \leq \frac{1}{\alpha} -1,$$
i.e., $r$ is bounded. Now, Lemma \ref{bd} implies that $\widetilde{S}$ contains only finitely many ($-1$)-curves.
%Indeed, if $f^{*}C - rE$ and $f^{*}C' - rE$ are ($-1$)-curves on $\widetilde{S}$ for two distinct linearly equivalent integral curves $C$ and $C'$ on $S$, then $f^{*}C - rE$ and $f^{*}C' - rE$ are also linearly equivalent. Therefore, $f^{*}C - rE = f^{*}C' - rE$. Thus at most one integral curve in each class $[C]$ gives a $(-1)$-curve in $\widetilde{S}$.

Since the proof of \cite[Theorem 3.4]{AL11} works for arbitrary characteristic, we can conclude that every nef divisor on $\widetilde{S}$ is semiample. Alternatively, we can directly show the semiampleness of a nef divisor using Lemma \ref{nefsemiample}. Therefore, $\widetilde{S}$ is a Mori dream rational surface.
\end{proof}

We are ready to prove Theorem \ref{redfg}.

\begin{proof}[Proof of Theorem \ref{redfg}]
Let $S$ be a Mori dream rational surface, and let $f \colon \widetilde{S} \rightarrow S$ be a redundant blow-up at a point $p$ with the exceptional divisor $E$. Note that $\kappa(-K_{\widetilde{S}})=\kappa(-K_S)$. Since every big anticanonical rational surface is always a Mori dream space, there is nothing to prove for the case $\kappa(-K_S)=2$. The case $\kappa(-K_S)=1$, the assertion is a subcase of Theorem \ref{1bwfg}. Thus it only remains the case $\kappa(-K_{S})=0$.

Let $-K_{\widetilde{S}} = \widetilde{P} + \widetilde{N}$ and $-K_S = P + N$ be the Zariski decompositions. We have $P=0$ since $P$ is semiample and $\kappa(P)=0$. By Lemma \ref{redlem}, $\widetilde{P}=f^{*}P=0, \text{ and hence, } -K_{\widetilde{S}}=\widetilde{N}=f^{*}N-E$. By the assumption $m:=\mult_p N > 1$, we have
$$\widetilde{N} = f^{-1}_{*} N + (m-1)E,$$
i.e., the redundant curve $E$ is contained in the support of $\widetilde{N}$.

First, we show that the effective cone $\Eff(\widetilde{S})$ is rational polyhedral. By Lemma \ref{cone}, it is sufficient to show that there are only finitely many ($-1$)-curves and ($-2$)-curves on $\widetilde{S}$. Note that every integral curve on $\widetilde{S}$ different from $E$ can be written as $\widetilde{C}=f^{*}C - rE$, where $C$ is an integral curve on $S$ and $r \geq 0$ is an integer. Note that $C^2 =r^2 -1$ and $K_S.C = -r-1$. Let $\widetilde{C}=f^{*}C - rE$ be a $(-2)$-curve which is not in the support of $\widetilde{N}$. Then we have
$$0=-K_{\widetilde{S}}.\widetilde{C} =\widetilde{N}.\widetilde{C}= (f^{-1}_{*} N + (m-1)E). \widetilde{C} \geq (m-1)E. \widetilde{C},$$
and hence, $E.\widetilde{C} = 0$. Since $E.\widetilde{C}=r$, we obtain $r=0$, i.e., $C$ is a ($-2$)-curve. Thus there are finitely many ($-2$)-curves on $\widetilde{S}$.

Let $\widetilde{C}=f^{*}C-rE$ be a ($-1$)-curve not in the support of $\widetilde{N}$. Then we have
$$1=-K_{\widetilde{S}}. \widetilde{C} = \widetilde{N}. \widetilde{C} = (f^{-1}_{*} N + (m-1)E). \widetilde{C} \geq (m-1)E. \widetilde{C} = (m-1)r,$$
and hence, $r \leq \frac{1}{m-1}$, i.e., $r$ is bounded. By Lemma \ref{bd}, there are only finitely many ($-1$)-curves on $\widetilde{S}$.

Finally, we show that every nef divisor $\widetilde{M}$ on $\widetilde{S}$ is semiample. By Lemma \ref{nefsemiample}, any nef divisor $\widetilde{M}$ with $-K_{\widetilde{S}}.\widetilde{M}>0$ is semiample. Since $-K_{\widetilde{S}}$ is effective, we may assume that $-K_{\widetilde{S}}.\widetilde{M} = 0$. We have
$$0=-K_{\widetilde{S}}.\widetilde{M}=\widetilde{N}. \widetilde{M} =  (f^{-1}_{*} N + (m-1)E). \widetilde{M} \geq (m-1)E. \widetilde{M},$$
and hence, we obtain $E.\widetilde{M}=0$. We may write $\widetilde{M} = f^{*}M - rE$ for some divisor $M$ on $S$. Then $r=E.\widetilde{M}=0$. For any effective curve $C$ on $S$, we have
$$M.C = f^* M. f^* C  = \widetilde{M}.f^{*}C \geq 0,$$
and hence, $M$ is also nef on a Mori dream surface $S$; therefore, it is semiample. Thus $\widetilde{M}=f^{*}M$ is also semiample.
\end{proof}

\begin{remark}
When $\kappa(-K_{S})=0$, the exceptional curve $E$ appearing in the support of $\widetilde{N}$ plays a crucial role in proving Theorem \ref{redfg}. If $E$ does not appear in the support of $\widetilde{N}$, then some nef divisor on $\widetilde{S}$ is not semiample and $\Eff(\widetilde{S})$ is not rational polyhedral (see Example \ref{nonmdrsex} and Remark \ref{nonmdrsrmk}). This observation naturally raise the question of which additional conditions guarantee that the redundant blow-up preserves the finite generation of Cox rings.
\end{remark}

\section{Examples of Mori dream rational surfaces with $\kappa(-K) \leq 0$}\label{exs}

In this section, we first construct Mori dream rational surfaces with $\kappa(-K) \leq 0$ admitting effective $k^{*}$-action. In particular, we prove Theorem \ref{exthm}. For generalities on $k^{*}$-surfaces, see \cite{OW77} and \cite{HS10}.

\begin{example}\label{-infty}
We can give an effective $k^{*}$-action on $\P^2 = \Proj k[x,y,z]$ by $t \cdot (x,y,z) = (t x, t y , z)$ for $t \in k^{*}$. Consider $k^{*}$-invariant lines $l_1 = V(x), l_2 = V(y), l_3 = V(x+y)$, and $L_1 = V(x-y), L_2 = V(z)$ in $\P^2$. Let $\phi \colon S \rightarrow \P^2$ be the blow-up at the base points of the cubic pencil determined by $xy(x+y)$ and $(x-y)z^2$. Then we get an elliptic fibration $\pi \colon S \rightarrow \P^1$ with two singular fibers of the same type $\widetilde{D}_4$. Throughout this example, we use the same notations for strict transforms of irreducible divisors.
%In the left side of the following figure, the solid lines are $(-2)$-curves and the dashed lines are $(-1)$-curves on $S$.

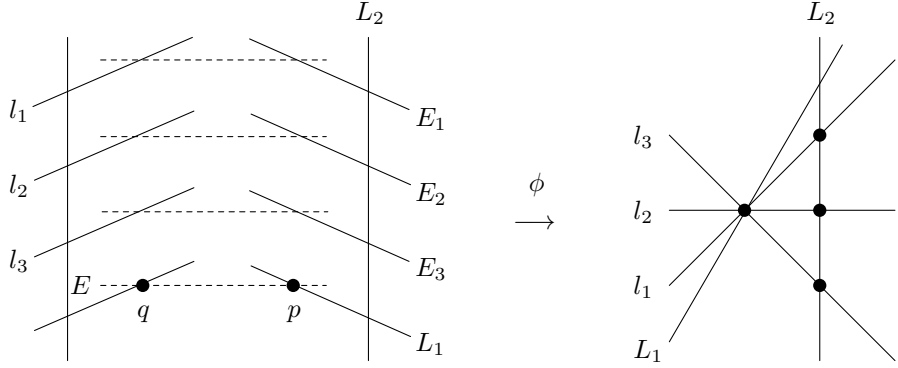
\begin{tikzpicture}[line cap=round,line join=round,>=triangle 45,x=1.0cm,y=1.0cm]
\clip(-3.9,-0.5) rectangle (8.5,5.3);
\draw (-3,4.3)-- (-3,0);
\draw (1,4.3)-- (1,0);
\draw (-3.46,3.38)-- (-1.34,4.3);
\draw [dash pattern=on 2pt off 2pt] (-2.56,4)-- (0.44,4);
\draw [dash pattern=on 2pt off 2pt] (-2.56,2.98)-- (0.44,2.98);
\draw [dash pattern=on 2pt off 2pt] (-2.54,1.98)-- (0.46,1.98);
\draw [dash pattern=on 2pt off 2pt] (-2.56,1)-- (0.44,1);
\draw (-3.44,2.4)-- (-1.32,3.32);
\draw (-3.44,1.38)-- (-1.32,2.3);
\draw (-3.44,0.4)-- (-1.32,1.32);
\draw (-0.58,4.28)-- (1.54,3.34);
\draw (-0.56,3.28)-- (1.56,2.34);
\draw (-0.58,2.26)-- (1.54,1.32);
\draw (-0.56,1.26)-- (1.56,0.32);
\draw (-2.2,0.85) node[anchor=north west] {$q$};
\draw (-0.2,0.85) node[anchor=north west] {$p$};
\draw (0.7,4.9) node[anchor=north west] {$L_2$};
\draw (1.5,0.49) node[anchor=north west] {$L_1$};
\draw (1.5,1.49) node[anchor=north west] {$E_3$};
\draw (1.5,2.49) node[anchor=north west] {$E_2$};
\draw (1.5,3.49) node[anchor=north west] {$E_1$};
\draw (-3.9,3.6) node[anchor=north west] {$l_1$};
\draw (-3.9,2.6) node[anchor=north west] {$l_2$};
\draw (-3.9,1.6) node[anchor=north west] {$l_3$};
\draw (-3.09,1.25) node[anchor=north west] {$E$};
\draw (2.8,2) node[anchor=north west] {$\longrightarrow$};
\draw (3.01,2.63) node[anchor=north west] {$\phi$};
\draw (5,1)-- (8,4);
\draw (5,3)-- (8,0);
\draw (5,2)-- (8,2);
\draw (5,0.25)-- (7.3,4.2);
\draw (7,4.3)-- (7,0);
\draw (6.7,4.9) node[anchor=north west] {$L_2$};
\draw (4.4,0.4) node[anchor=north west] {$L_1$};
\draw (4.4,1.25) node[anchor=north west] {$l_1$};
\draw (4.4,2.25) node[anchor=north west] {$l_2$};
\draw (4.4,3.25) node[anchor=north west] {$l_3$};
\begin{scriptsize}
\fill [color=black] (-2,1) circle (2.5pt);
\fill [color=black] (0,1) circle (2.5pt);
\fill [color=black] (6,2) circle (2.5pt);
\fill [color=black] (7,3) circle (2.5pt);
\fill [color=black] (7,2) circle (2.5pt);
\fill [color=black] (7,1) circle (2.5pt);
\end{scriptsize}
\end{tikzpicture}

If we blow-up at a $k^{*}$-fixed point, then $k^{*}$-action uniquely extends (\cite[Lemma 1.8]{OW77}). Thus $k^{*}$-action on $\P^2$ extends to $S$. Since, smooth projective rational surfaces with effective $k^{*}$-action are Mori dream surfaces (\cite{Kn}), $S$ is a Mori dream rational surface. There is a section $E$ of the elliptic fibration $\pi \colon S \rightarrow \P^1$ meeting two singular fibers at $p$ and $q$, respectively. Since any negative self-intersection curve is $k^{*}$-invariant (\cite[Proposition 1.9]{OW77}), $p$ and $q$ are $k^{*}$-invariant.

Let $f(0) \colon S_p(0) \rightarrow X$ be the blow-up at $p$. Note that $-K_S = 2L_2 + L_1 + E_1 + E_2 + E_3$. The Zariski decomposition $-K_{S_p(0)}=P+N$ is given by $P=0$ and $N=2L_2 + L_1 + E_1 + E_2 + E_3$, and hence, $\kappa(-K_{S_p(0)})=0$. Since the intersection points of negative self-intersection curves are $k^{*}$-invariant points, there is a $k^{*}$-invariant point on some component of $N$. Let $f(1) \colon S_p(1) \rightarrow S_p(0)$ be the blow-up at a $k^{*}$-invariant point on $N$. Then $S_p(1)$ has $k^{*}$-action, and $f(1)$ is a redundant blow-up. We can repeat this construction $f(n) \colon S_p(n) \rightarrow S_p (n-1)$ for each $n \geq 2$. Note that $S_p(n)$ is a Mori dream rational surface with $\kappa(-K_{S_p(n)})=0$ of Picard number $n+11$ for each $n \geq 0$.

For each $n \geq 0$, let $g(n) \colon \widetilde{S}_q(n) \rightarrow S_p(n)$ be the blow-up at the $k^{*}$-fixed point $q$. Note that there is a birational surjective morphism $h(n) \colon\widetilde{S}_q(n) \rightarrow \widetilde{S}_q (n-1)$. We claim that $\kappa(-K_{\widetilde{S}_q(n)})=-\infty$. Indeed, if $\kappa(-K_{\widetilde{S}_q(n)})=0$, then $g(n)$ is a redundant blow-up by Theorem \ref{chbl}. However, for the Zariski decomposition $-K_{S_p(n)} = P(n) + N(n)$, the point $q$ is not in the support of $N(n)$, and hence, $\mult_q N(n)=0$. This is a contradiction. Thus $\widetilde{S}_q(n)$ is a Mori dream rational surface with $\kappa(-K_{\widetilde{S}_q(n)})=-\infty$ of Picard number $n+12$ for each $n \geq 0$.

\[
 \xymatrix{
 \cdots \ar[r]^-{h(2)} & \ar[d]^-{g(1)} \widetilde{S}_q(1) \ar[r]^-{h(1)}&\ar[d]^-{g(0)} \widetilde{S}_q(0) & & \\
\cdots \ar[r]^-{f(2)} & S_p(1) \ar[r]^-{f(1)} & S_p(0) \ar[r]^-{f(0)} & \ar[d]_-{\pi} S \ar[r]^-{\phi} & \P^2\\
 & & & \P^1 &
}
\]
\end{example}

\begin{remark}\label{-infrem}
Note that there exists a $k^*$-invariant redundant point $p'$ on $S_p(n)$ with $\mult_{p'} N(n)=1$ for $n \geq 0$ so that the surface obtained by blowing up at the point $p'$ is a Mori dream space.
\end{remark}

\begin{remark}\label{extell}
We assume that $k=\C$. If we contract $E$ on $S$, then we get the minimal resolution of a degree one del Pezzo $\C^{*}$-surface of Picard number one with two canonical singularities of the same type $D_4$. Del Pezzo $\C^{*}$-surfaces of Picard number one with canonical singularities are completely classified (\cite[Theorem 5.6]{HS10}). From this classification, we can easily derive that there are only four extremal rational elliptic $\C^*$-surfaces $X_{11}, X_{22}, X_{33}$, and $X_{44}$ among the list in \cite[Theorem 4.1]{MP86}. We can carry out a similar construction in Example \ref{-infty} for these elliptic surfaces.
\end{remark}

%If the degree is one, then there are exactly four possible singularity types; $2D_4, E_6 \oplus A_2, E_7 \oplus A_1$, or $E_8$. On the other hand, $S$ is an extremal rational elliptic surface $X_{11}$ in the sense of Miranda and Persson (see \cite[Theorem 4.1]{MP86}). In general, for each minimal resolution $Y$ of degree one del Pezzo $\C^{*}$-surfaces of Picard number one with canonical singularities, there exists a $\C^{*}$-fixed point $y$ on $Y$ such that $X$ is an extremal rational elliptic surface, where $X \rightarrow Y$ is the blow-up at $y$. Note that the four extremal rational elliptic surfaces $X_{22}, X_{33}, X_{44}$, and $X_{11}$ in \cite[Theorem 4.1]{MP86} have an effective $\C^{*}$-action. Thus we can carry out a similar construction for these elliptic surfaces, as in Example \ref{-infty}.

Now, we construct a Mori dream rational surface $S$ with $\kappa(-K_S)=0$ admiting a redundant blow-up $\widetilde{S} \rightarrow S$ such that $\widetilde{S}$ is not a Mori dream space. To do this, we need the following easy lemma.

\begin{lemma}\label{ka=0K^2=0}
There is no Mori dream rational surface $S$ with $\kappa(-K_S)=0$ and $K_S^2=0$.
\end{lemma}

\begin{proof}
By the Riemann-Roch formula, we get $h^0(\mathcal{O}_S(-K_S)) \geq 1$, and hence, we can have the Zariski decomposition $-K_S=P + N$.  Note that $0 = (-K_S)^2 = P^2 + N^2$. Suppose that $N \neq 0$. Then $N^2<0$, and hence, $P$ and $-K_S$ is big, which is a contradiction. Thus we have $N=0$. Then $-K_S=P$ is nef. However, since $\kappa(-K_S)=0$, it follows that $-K_S$ is not semiample. In particular, $Y$ is not a Mori dream space.
\end{proof}

\begin{example}\label{nonmdrsex}
Consider an effective $k^*$-action on $\P^2 = \Proj k[x,y,z]$ by $t \cdot (x,y,z)=(t^3x, t^2y,z)$. This action preserves the cubic pencil determined by $z^3$ and $x^2z + y^3$. Let $X_{22} \to \P^2$ be the blow-up at the base points of the cubic pencil. Then $X_{22}$ is a relatively minimal rational elliptic surface with two singular fibers; one is a cuspidal rational curve $C$ and the other one consists of nine $(-2)$-curves forming the dual graph $\widetilde{E}_8$. The elliptic fibration of $X_{22}$ has a unique section $E_s$ meeting $C$ at a point $p$. Let $X_p \rightarrow X_{22}$ be the blow-up at $p$ with the exceptional divisor $E_p$. Then $-K_{X_p} = 0 + C$ is the Zariski decomposition, and hence, $\kappa(-K_{X_p})=0$. Throughout this example, we use the same notations for strict transforms. Note that $X_{22}$ admits an effective $k^*$-action and $p$ is a $k^*$-invariant point. Thus $X_p$ is a Mori dream rational surface (\cite{Kn}). Note that every point on $C$ is a redundant point on $X_p$. Let $q$ be a point on $C$ different from $k^*$-invariant points, and let $X_{pq} \rightarrow X_p$ be the redundant blow-up at $q$ with the exceptional divisor $E_q$.

\begin{claim}
$X_{pq}$ is not a Mori dream space.
\end{claim}

To prove the claim, we contract the curve $E_p$ and then the curve $E_s$ so that we obtain a composition of blow-downs $X_{pq} \rightarrow X_q \rightarrow X'$. It suffices to show that $X'$ is not a Mori dream space. Contracting $E_s$ on $X_{22}$ or $E_q$ on $X'$, we obtain a weak del Pezzo surface $X$.
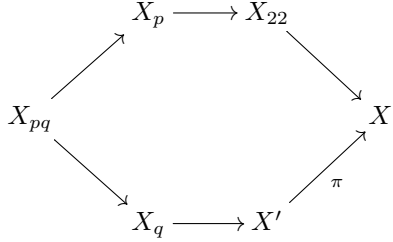
\[
 \xymatrix{
&X_p \ar[r] & X_{22} \ar[rd] &\\
X_{pq} \ar[ru] \ar[rd] &&& X\\
&X_q \ar[r] & X' \ar[ru]_{\pi}&
}\]
Note that $-K_{X'} = \pi^{*}(-K_X) - E_q = C$. Since $K_{X'}^2 = C^2 = 0$, it follows that $-K_{X'}$ is nef, thus, $\kappa(-K_{X'})=0$ or $1$. First, we show that $\kappa(-K_{X'})=0$. Note that $\dim |-K_X|=1$. By considering the elliptic fibration on $X_{22}$, we see that $C$ is the unique curve in $|-K_X|$ passing through $q$. Thus $|-K_{X'}|$ has one element $C$, and hence, $h^0(\mathcal{O}_{X'}(C))=1$. From the following exact sequence
$$
0 \rightarrow \mathcal{O}_{X'} \rightarrow \mathcal{O}_{X'}(C) \rightarrow \mathcal{O}_C(C) \rightarrow 0,
$$
we obtain $h^0(\mathcal{O}_C(C))=0$. Note that the group of Cartier divisors of degree zero on the cuspidal rational curve $C$ is isomorphic to the additive group $\textbf{G}_a$. Since there is no torsion element in $\textbf{G}_a$, we obtain $h^0(\mathcal{O}_C(mC))=0$ for any integer $m>0$. By twisting the above exact sequence by $\mathcal{O}_{X'}(mK_{X'})$, we get
$$
h^0(\mathcal{O}_{X'}(-(m+1)K_{X'}))=h^0(\mathcal{O}_{X'}(-mK_{X'}))= \cdots = h^0(\mathcal{O}_{X'}(-K_{X'}))=1
$$
for all integers $m>0$, i.e., $\kappa(-K_{X'})=0$. By Lemma \ref{ka=0K^2=0}, $X'$ is not a Mori dream space.
\end{example}

\begin{remark}\label{nonmdrsrmk}
Note that $-K_{X'}$ is nef but not semiample. Thus the pull-back of $-K_{X'}$ to $X_{pq}$ is nef but not semiample. Furthermore, the effective cone of $X_{pq}$ is not rational polyhedral. To see this, note that the dual graph of $(-2)$-curves on $X'$ does not form one of $H\widetilde{E}_8, H\widetilde{D}_8$ or $H\widetilde{A}_8$. By Nikulin's classification of surfaces $(***)$ in \cite[p.84]{Nik00}, the effective cone of $X'$ is not rational polyhedral.
\end{remark}

\section{Minimal resolutions of rational $\mathbb{Q}$-homology projective planes}\label{qhpps}

In this section, we investigate the relation between Mori dream rational surfaces and minimal resolutions of rational $\Q$-homology projective planes. First, we prove Theorem \ref{minen}.

\begin{proof}[Proof of Theorem \ref{minen}]
Recall that $\bar{S}$ is a rational surface with log terminal singularities having Picard number one, and $g \colon S \rightarrow \bar{S}$ is the minimal resolution. Since $-K_{\bar{S}}$ is nef, we have two cases: $-K_{\bar{S}}$ is ample or $-K_{\bar{S}}$ is numerically trivial. If $-K_{\bar{S}}$ is ample, then $S$ is a big anticanonical rational surface. Thus the assertion immediately follows.

Now, assume that $-K_{\bar{S}}$ is numerically trivial and $\bar{S}$ does not contain any canonical singularity. The Zariski decomposition $-K_S = P +N $ is given by
$$P=g^* (-K_{\bar{S}})=0 \hbox{ and } N=\sum_{i=1}^{l} a_i E_i, $$
where each $E_i$ denotes an irreducible component of the $g$-exceptional divisor, and we have $0 < a_i < 1$ for all $i$.
%Note that
%$$\Pic(S)=g^{*}(\Pic(\bar{S})) \oplus \bigoplus_{i=1}^l \Z \cdot E_i,$$
%and by the assumption, we have $0 < a_i < 1$ for all $i$.

First, we show that the effective cone $\Eff(S)$ is rational polyhedral. By Lemma \ref{cone}, it suffices to show that there are only finitely many ($-1$)-curves and ($-2$)-curves on $S$.  Let $H$ be an ample generator of $\Pic(\bar{S})$. We may write every integral curve not in the support of $N$ as
$$C = g^* (bH) + \sum_{i=1}^l b_i E_i.$$
Note that $b \geq 0$.

Suppose that $C$ is a ($-1$)-curve. Let $\lambda_i := C.E_i$ be a nonnegative integer for each $i$. Then we have
$$1=C.(-K_S) = C.N = \sum_i a_i C.E_i = \sum_i a_i \lambda_i.$$
Since each $a_i$ is a positive rational number, there are finitely many possibilities for $(\lambda_1, \ldots, \lambda_l)$. Now, we obtain
$$\lambda_j = C.E_j=g^{*}(bH).E_j + \sum_i b_i E_i.E_j=\sum_i b_i E_i.E_j.$$
Since the intersection matrix of irreducible components of $N$ is negative definite, $(b_1, \ldots, b_l)$ is determined when $(\lambda_1, \ldots, \lambda_l)$ is given.  Moreover, we have
$$-1 = C^2 = f^*(b H)^2  + \left( \sum_i b_i E_i \right)^2.$$
Thus $\pi^*(bH)^2$ is determined by $(b_1, \ldots, b_l)$, so is $b$. We have shown that there are finitely many possibilities for $b$ and $(b_1, \ldots, b_l)$. This means that there are only finitely many ($-1$)-curves on $S$.
%we obtain the matrix equation
%\begin{displaymath}
%\left( \begin{array}{cccc}
%E_1^2 & E_2 E_1 & \cdots & E_l E_1 \\
%E_2 E_1 &E_2^2 & \cdots & E_l E_2 \\
%\vdots & \vdots & \ddots & \vdots \\
%E_l E_1 & E_l E_2 & \cdots & E_l^2
%\end{array} \right)
%\left( \begin{array}{c} b_1 \\ b_2 \\\vdots \\ b_l \end{array} \right) =
%\left( \begin{array}{c} \lambda_1 \\ \lambda_2 \\\vdots \\ \lambda_l \end{array} \right)
%\end{displaymath}
%which has a unique solution $(b_1, b_2, \ldots, b_l)$, because the intersection matrix on the left-hand side is negative definite.

Suppose now that $C$ is a $(-2)$-curve not in the support of $N$. Then we have
$$-K_S.C=N.C=\sum_i a_i E_i.C = 0,$$
and hence, $E_i.C=0$ for all $i$. Thus we may write $C=g^*(bH)$ for some $b \geq 0$, so $C^2 \geq 0$, which is a contradiction. Hence, every $(-2)$-curve is contained in the support of $N$.

It only remains to show that every nef effective $\Q$-divisor $M$ is semiample. For sufficiently small $\epsilon>0$, the pair $(S, N+ \epsilon M)$ is Kawamata log terminal, and $K_S + N + \epsilon M = \epsilon M$ is nef. By the log abundance theorem for surfaces (see \cite{FM91} for $\cha(k)=0$ and \cite{Ta} for $\cha(k)>0$), $M$ is semiample. We complete the proof.
\end{proof}

\begin{remark}
Since $\bar{S}$ does not contain any canonical singularities, every $g$-exceptional curve appears in $N$. It plays an important role in showing that $S$ contains finitely many ($-1$)-curves. There exists a rational surface with numerically trivial anticanonical divisor containing canonical singularities (see \cite[Section 6]{HKex}). We do not know whether the minimal resolution of such a surface is a Mori dream rational surface.
\end{remark}

%\begin{remark}
%$(S, N)$ is a \emph{Calabi-Yau pair}, which is a Kawamata log terminal pair such that $K_S+N$ is numerically trivial, in the sense of \cite{Tot10}.
%\end{remark}

If the anticanonical divisor $-K_{\bar{S}}$ of a normal projective rational surface $\bar{S}$ of Picard number one is numerically trivial, then the anticanonical Iitaka dimension $\kappa(-K_S)=0$ of the minimal resolution $S$, since $P=0$ for the Zariski decomposition $-K_S = P+N$.

In the remainder of the section, we assume that $k=\C$. Recall that a \emph{rational $\mathbb{Q}$-homology projective plane} $\bar{S}$ is, by definition, a complex normal projective rational surface with at worst quotient singularities and the second topological Betti number $b_2(\bar{S}) = 1$. The following is an immediate consequence of Theorem \ref{minen}.

\begin{corollary}\label{qhcor}
Let $\bar{S}$ be a rational $\Q$-homology projective plane, and let $g \colon  S \rightarrow \bar{S}$ be its minimal resolution. Assume that $-K_{\bar{S}}$ is nef and $\bar{S}$ does not contain any rational double point. Then $S$ is a Mori dream space.
\end{corollary}

Using Theorem \ref{minen}, we construct examples of Mori dream rational surfaces with anticanonical Iitaka dimension $0$.

\begin{example}\label{0ex}
In Section 4 of \cite{HKex}, the construction of the minimal resolution $Z:=Z(3,3,3,3)$ of a rational $\Q$-homology projective planes $T(3,3,3,3)$ with numerically trivial anticanonical divisor having two cyclic quotient singularities of the same type $[2,2,3,3,2,2]$ is given. In Section 5 of \cite{HKex}, the minimal resolution $Z(2)$ of rational $\Q$-homology projective plane $S(2)$ with numerically trivial anticanonical divisor containing a unique cyclic quotient singularity of type $[3,2,2,2,2,2,2,2,2,3]$ is also explicitly constructed. Note that $\kappa(-K_{Z(3,3,3,3)})=\kappa(-K_{Z(2)})=0$
\end{example}

%The construction in Section 4 of \cite{HKex} is as follows. Let $\overline{l}_1, \overline{l}_2, \overline{l}_3, \overline{l}_4$ be four general lines on $\P^2$, and choose four intersection points. Let $Z(2,2,2,2) \rightarrow \P^2$ be the blow-up at the four intersection points, twice each. Let $E_1,E_2,E_3,E_4$ be the four exceptional curves of the first kind. We may assume that each $E_i$ meets the strict transform $l_i$ of $\overline{l}_i$ for $i=1,2,3,4$. By blowing-up the four points $E_i \cap l_i$ for $i=1,2,3,4$ in $Z(2,2,2,2)$, we get $Z(3,3,3,3)$, which is the minimal resolution of a rational $\Q$-homology projective plane $T(3,3,3,3)$ with numerically trivial anticanonical divisor. Moreover, $T(3,3,3,3)$ has two cyclic quotient singularities of the same type $[2,2,3,3,2,2]$.

\begin{remark}\label{ncob}
Each Mori dream rational surface in Example \ref{0ex} is obtained by blow-ups of a big anticanonical rational surface at one point. They are different from the Mori dream rational surface $\widetilde{Y}$ with $\kappa(-K_{\widetilde{Y}})=0$ in \cite[Section 6]{LT11}, because the Zariski decompositions of anticanonical divisors are different. On the other hand, $\widetilde{Y}$ and $Z(2)$ are Coble rational surfaces in the sense of Dolgachev and Zhang (\cite{DZ01}), but $|-40K_{Z(3,3,3,3)}|$ is nonempty and $|-mK_{Z(3,3,3,3)}|$ is empty for $0<m<40$. Hence, $Z(3,3,3,3)$ is not a Coble rational surface.
\end{remark}

Mori dream rational surfaces in Example \ref{0ex} have redundant points by \cite[Theorem 1.2]{HP}. Thus we get more Mori dream rational surfaces with anticanonical Iitaka dimension $0$. Furthermore, we obtain the following (cf. \cite[Theorem 1.4]{HP} and \cite[Remark 3]{TVV10}).

\begin{proposition}\label{nominen}
There exists a Mori dream rational surface $\widetilde{Z}$ with $\kappa(-K_{\widetilde{Z}})$$=0$ such that $(\widetilde{Z}, -K_{\widetilde{Z}})$ is a log Calabi-Yau pair (in the sense of \cite{Tot10}) but not a minimal resolution of a normal projective rational surface $S$ with rational singularities such that $-K_S$ is nef.
\end{proposition}

\begin{proof}
Consider the rational surface $Z:=Z(3,3,3,3)$ in Example \ref{0ex}. Let $-K_Z = P+N$ be the Zariski decomposition. By simple calculation, we can easily see that the intersection $p$ of two ($-3$)-curves in $Z$ is a redundant point and $\mult_p N= \frac{3}{2}$. Let $f \colon  \widetilde{Z} \rightarrow Z$ be the redundant blow-up at $p$, and let $-K_{\widetilde{Z}} = \widetilde{P}+\widetilde{N}$ be the Zariski decomposition. By Theorem \ref{redfg}, $\widetilde{Z}$ is a Mori dream rational surface with $\kappa(-K_{\widetilde{Z}})=0$, and $(\widetilde{Z}, -K_{\widetilde{Z}})$ is a log Calabi-Yau pair. The exceptional curve $E$ of $f$ is contained in the support of $\widetilde{N}$ by Lemma \ref{redlem}. Suppose now that $\widetilde{Z}$ is the minimal resolution of a normal projective rational surface $S$ with rational singularities such that $-K_S$ is nef. Then $\widetilde{N}$ consists of only exceptional divisors on the minimal resolution. Thus $\widetilde{N}$ does not contain a ($-1$)-curve, which is a contradiction.
\end{proof}

Furthermore, we can get an example of Mori dream rational surface with $\kappa(-K)=0$ whose contractions are all big anticanonical rational surfaces.

\begin{proposition}\label{0to2}
There is a Mori dream rational surface $X$ with $\kappa(-K_X)=0$ such that for any blow-down $X \rightarrow Y$ contracting one $(-1)$-curve, we have $\kappa(-K_Y)=2$.
\end{proposition}

\begin{proof}
Let $X:=Z(2)$ be the surface in Example \ref{0ex}, and let $-K_X = 0 + N$ be the Zariski decomposition. The Picard number of $X$ is 11, and the number of irreducible components of $N$ is 10. Let $f \colon X \rightarrow Y$ be a blow-down contracting one $(-1)$-curve $E$.
Suppose that $\kappa(-K_Y) \leq 1$. By Lemma \ref{ka=0K^2=0}, we only have to consider the case $\kappa(-K_Y)=1$. Then $Y$ is a relatively minimal rational elliptic surface. The elliptic fibration of $Y$ is given by $|-mK_Y|$ for some integer $m>0$. There is a unique fiber $F$ of the elliptic fibration of $Y$ passing through the point $p=f(E)$, and $\mult_p(F)=m$. Then we have
$$
-K_X = f^*(-K_Y)-E = \frac{1}{m}f^*(F) - E = \frac{1}{m}f^{-1}_{*}F
$$
so that $F$ has 10 irreducible components.
Thus the Jacobian fibration of $Y$ has a fiber containing 10 irreducible components. However, there is no extremal rational elliptic surface with a singular fiber containing 10 irreducible components (see \cite[Theorem 4.1]{MP86}). It is a contradiction, and hence, $\kappa(-K_Y)=2$.
\end{proof}

On the other hand, there exist infinitely many rational $\mathbb{Q}$-homology projective planes with ample canonical divisors (see \cite{HKex}). However, we do not know whether the Cox rings of minimal resolutions of those surfaces are finitely generated.

\begin{question}\label{qhpp}
Let $\bar{S}$ be a rational $\mathbb{Q}$-homology projective plane, and let $S$ be its minimal resolution. Is the Cox ring of $S$ finitely generated?
\end{question}

\begin{remark}\label{Ohashi}
The rationality assumption in Question \ref{qhpp} is essential. Indeed, there is an Enriques surface with infinite automorphism group having nine $(-2)$-curves forming $D_8 \oplus A_1$ or $E_7 \oplus 2A_1$ (\cite{HKO}). Note that there are infinitely many $(-2)$-curves. In particular, the Cox rings of those Enriques surfaces are not finitely generated. However, by contracting those nine curves, we obtain a $\Q$-homology projective plane and the contracting map is a minimal resolution.
% the Cox ring of an Enriques surface with infinite automorphism group is not finitely generated.
\end{remark}

%BIBLIOGRAFIA

\end{document}